\documentclass[11pt]{amsart}
\usepackage{amsmath,amssymb,amsthm,amscd,verbatim}
\usepackage{graphicx}
\usepackage{epsfig}
\usepackage{hyperref}
\usepackage{xcolor}

\newtheorem{thm}{Theorem}[section]
\newtheorem{lem}[thm]{Lemma}

\newtheorem{obs}[thm]{Observation}

\newtheorem*{thm*}{Theorem}

\theoremstyle{definition}

\theoremstyle{remark}

\newtheorem*{acknowledgement}{Acknowledgments}

\newcommand{\poly}{\ensuremath{\mathrm{poly}}}
\newcommand{\de}{\ensuremath{\mathrm{deg}}}

\newcommand{\Exp}{\,\mathbb{E}}
\renewcommand{\Pr}{\,\mathbb{P}}

\renewcommand{\le}{\leqslant}
\renewcommand{\leq}{\leqslant}
\renewcommand{\ge}{\geqslant}
\renewcommand{\geq}{\geqslant}

\def\longequation{$$\vcenter\bgroup\advance\hsize by -9em%
\noindent\ignorespaces\refstepcounter{equation}}%
\makeatletter%
\def\endlongequation{\egroup\eqno(\theequation)$$\global\@ignoretrue}
\makeatother

\begin{document}
\title[Distributed coloring of graphs with an optimal number of colors]{Distributed coloring of graphs\\ with an optimal number of colors}

\author{\'Etienne Bamas} \address{School of Computer and Communication
  Sciences, \'Ecole Polytechnique F\'ed\'erale de Lausanne, Switzerland}
\email{etienne.bamas@epfl.ch}

\author{Louis Esperet} \address{Laboratoire G-SCOP (CNRS, Univ. Grenoble Alpes), Grenoble, France}
\email{louis.esperet@grenoble-inp.fr}

\thanks{An extended abstract of this work  appeared in the proceedings
  of the 36th International Symposium on Theoretical Aspects of Computer Science (STACS), 2019.\newline Partially supported by ANR Project GATO (\textsc{anr-16-ce40-0009-01}) and LabEx PERSYVAL-Lab
  (\textsc{anr-11-labx-0025}).}

\date{}
\sloppy

\begin{abstract}
  This paper studies sufficient conditions to obtain efficient
distributed algorithms coloring graphs optimally (i.e.\ with the
minimum number of colors) in the \textsf{LOCAL} model of
computation. Most of the work on distributed vertex coloring so far has focused on
coloring graphs of maximum degree $\Delta$ with at most $\Delta+1$
colors (or $\Delta$ colors when some simple obstructions are forbidden).
When $\Delta$ is sufficiently large and $c\ge
\Delta-k_\Delta+1$, for some integer $k_\Delta\approx
\sqrt{\Delta}-2$, we give a distributed algorithm that given a
$c$-colorable graph
$G$ of maximum degree $\Delta$, finds a $c$-coloring of
$G$ in $\min\{O((\log\Delta)^{1/12}\log n), 2^{O(\log
  \Delta+\sqrt{\log \log n})}\}$ rounds, with high probability. The lower bound $\Delta-k_\Delta+1$ is best possible in the sense that
for infinitely many values of $\Delta$, we prove that when
$\chi(G)\le \Delta -k_\Delta$, finding an optimal
coloring of $G$ requires $\Omega(n)$ rounds. Our proof is a light
adaptation of a remarkable result of Molloy and Reed, who proved that
for $\Delta$ large enough, for any $c\ge \Delta-
k_\Delta$ deciding whether $\chi(G)\le c$ is in {\textsf{P}}, while
Embden-Weinert \emph{et al.}\ proved that
for $c\le \Delta-k_\Delta-1$, the same problem is
{\textsf{NP}}-complete. Note that the sequential and distributed thresholds
differ by one. 

Our first result covers the case where the chromatic number of the graph
ranges between $\Delta-\sqrt{\Delta}$ and $\Delta+1$. Our second
result covers a larger range, but gives a weaker bound on the
number of colors:  For any sufficiently large $\Delta$, and
$\Omega(\log \Delta)\le k \le \Delta/100$, we prove that every graph of maximum
degree $\Delta$ and clique number at most $\Delta-k$ can
be efficiently colored with at most $\Delta-\varepsilon k$ colors, for
some absolute constant $\varepsilon >0$, with a randomized algorithm
running in $O(\log n/\log \log n)$ rounds with high probability.
\end{abstract}
\maketitle

\section{Introduction}

The graph coloring problem plays an important role in
distributed computing, since it is
used as a subroutine in distributed algorithms for
a large variety of problems (see the recent survey book of Barenboim and
Elkin~\cite{BE13} for more details and further references). The central problem in distributed coloring is the $(\Delta+1)$-coloring problem, where a
graph of maximum degree at most $\Delta$ has to be colored with at
most $\Delta+1$ colors (see~\cite{FHK16} and~\cite{CLP18} for the fastest deterministic and randomized
algorithms to date). The bound $\Delta+1$ on the number of colors
is best possible in general, but it
follows from Brooks' Theorem that any connected graph of maximum
degree $\Delta$ which is neither an odd cycle nor a complete graph can indeed be colored with
$\Delta$ colors, instead of $\Delta+1$, and there has been some work to find fast distributed algorithms
coloring such graphs with
$\Delta$ colors. The problem was first considered by Panconesi and
Srinivasan~\cite{PS95}, and it was recently proved in~\cite{GHKM18}
that the $\Delta$-coloring problem can be solved with a randomized algorithm running in
$O(\log \Delta)+2^{O(\sqrt{\log \log n})}$ rounds when $\Delta\ge 4$,
or $O((\log \log n)^2)$ rounds when $\Delta$ is a constant. On the
other hand, it was proved in~\cite{BFH18} that a randomized algorithm
solving the $\Delta$-coloring problem needs $\Omega(\log
\log n)$ rounds. These
results, as well as all the other algorithms mentioned in this paper,
are proved in the \textsf{LOCAL} model of computation (see below for
more details). 

The main idea of $\Delta$-coloring is that by forbidding some simple
obstructions (complete graphs and odd cycles), we can save one color
(compared with the easier $(\Delta+1)$-coloring problem)
while still having a fast algorithm, whether sequential or
distributed. A natural question is: can we go further? Is there some
small set of obstructions (that can be easily recognized locally, at least
when $\Delta$ is sufficiently large), such that if we forbid these
obstructions we can find fast distributed algorithms coloring graphs
of maximum degree $\Delta$ with $\Delta-1$ colors? Or $\Delta-2$ colors?
Or $\Delta-k$ colors, for some constant $k$?

The sequential version of this question turned out to have a very
precise answer.  For any $\Delta$, let $k_\Delta$ be the maximum integer $k$ such that
$(k+1)(k+2)\le \Delta$. It can be checked that $k_\Delta=\lfloor
\sqrt{\Delta+1/4}-3/2\rfloor$ and thus 
$\sqrt{\Delta}-3<k_\Delta < \sqrt{\Delta}-1$. The following was proved
by Embden-Weinert, Hougardy and Kreuter~\cite{EHK98}.

\begin{thm}[\cite{EHK98}]\label{thm:emb}
For $3 \le  c \le \Delta- k_\Delta -1$, we cannot test for
$c$-colorability of graphs with maximum degree $\Delta$ in polynomial
time unless {\textsf{P}} $=$ {\textsf{NP}}.
\end{thm}

The following strong converse was then proved by Molloy and Reed~\cite{MR14}.

\begin{thm}[\cite{MR14}]\label{thm:mr1}
For sufficiently large (but constant) $\Delta$, 
and every $c \ge \Delta -k_\Delta$, there is a linear time
deterministic algorithm to test whether graphs of maximum degree
$\Delta$ are $c$-colorable. Furthermore, there is a polynomial time
deterministic algorithm that will produce a $c$-coloring whenever one exists.
\end{thm}

Our main result will be to prove that a similar dichotomy occurs in
the \textsf{LOCAL} model, with a slightly larger tractability threshold
($\Delta -k_\Delta+1$ instead of $\Delta -k_\Delta$).

\begin{thm}\label{thm:col}
For sufficiently large $\Delta$, and any $c \ge \Delta-k_{\Delta}+1$, there is a distributed randomized
algorithm that
takes a graph $G$ with maximum degree $\Delta$ as input, and does the
following: either some vertex outputs a certificate
that $G$ is not $c$-colorable, or the algorithm finds a $c$-coloring
of $G$. The algorithm runs in $O(\textsf{T}_{\de+1}(n,\Delta)) +O((\log\Delta)^{13/12})\cdot \textsf{T}_{LLL}(n,\poly \,
\Delta)$ rounds w.h.p., which is
$\min\{O((\log\Delta)^{1/12}\log n), 2^{O(\log \Delta+\sqrt{\log \log
    n})}\}$ rounds w.h.p.
\end{thm}

Here, w.h.p.\ (with high probability) means with probability at least
$1-O(n^{-\alpha})$, for any fixed $\alpha>0$. The values of
$\textsf{T}_{LLL}$ and $\textsf{T}_{\de+1}$ correspond to the round
complexities of the distributed Lov\'asz Local Lemma and the
$(\deg+1)$-list coloring problem. The precise definitions  will be given in  Section~\ref{sec:prel}.

Note that the
chromatic number of $G$ can be smaller than the threshold
$\Delta-k_\Delta+1$, what matters is that the number $c$ of available
colors is at least this threshold. 
We will prove that the value of $\Delta-k_\Delta+1$ is sharp, in the
following sense.

\begin{thm}\label{thm:sharp}
When $c \le \Delta - k_\Delta - 1$ (for any value of $\Delta$), and
when $c= \Delta-k_\Delta$ (for infinitely many values of $\Delta$), there exist arbitrarily
large graphs $G$ of maximum degree $\Delta$ for which
$\chi(G)=c$, and such that any
distributed algorithm coloring $G$ with $c$ colors takes
$\Omega(n/\Delta)$ rounds.
\end{thm}

In the \textsf{LOCAL} model of computation, if
the algorithm runs in $r$ rounds, the color assigned to a vertex $v$ is based
only on the (subgraph induced by the) vertices at distance at most $r$
from $v$. The fact that when $c \ge \Delta-k_{\Delta}+1$, it can be
decided whether $G$ is $c$-colorable by only looking at each
neighborhood was already proved by Molloy and Reed~\cite{MR14} (see
Theorem~\ref{thm:mrlocal}). In this paper, we are mostly interested in
\emph{producing} such a coloring in a distributed way, and it is a priori unclear that it
can be done in a small number rounds. For instance, in the \textsf{LOCAL} model it can be decided
in a single round whether a graph has maximum
degree at most two (and is therefore 3-colorable), but finding a
3-coloring of a path takes an unbounded number of rounds~\cite{Lin92}.

An interesting difference between Theorems~\ref{thm:col} and~\ref{thm:mr1} (besides the fact
that the sequential and distributed thresholds are not the same), is
that in the sequential result it is crucial that $\Delta$ is a
constant. If $\Delta$ depends on $n$, then Molloy and
Reed~\cite{MR14} proved that the tractability threshold is around
$\Delta-\Theta(\log \Delta)$ colors. On the other hand, in the distributed setting there is no
requirement on $\Delta$.

\medskip

It should be mentioned that efficient distributed coloring algorithms involving the
chromatic number are not frequent. A rare example of such an algorithm involving a general class of graphs (not just paths or
cycles, or line-graphs for instance) is
the following result of Schneider and Wattenhofer~\cite{SW11}: when
$\Delta=\Omega(\log^{1+1/\log^*n}n)$ and $\chi=O(\Delta/\log^{1+1/\log^*n}n)$, they find a randomized distributed
algorithm coloring graphs of maximum degree $\Delta$ and chromatic
number $\chi$ with at most $(1-1/O(\chi))\Delta$ colors w.h.p., and
running w.h.p.\ in $O(\log \chi+\log^*n)$ rounds. Two significant
differences with our result are the requirement on $\Delta$ and
the fact that the number of colors in the resulting coloring is not
best possible. We also note that in the setting of Theorem~\ref{thm:mr1}
and Theorem~\ref{thm:1} below, the chromatic number is an additive factor away
from the maximum degree, while the result of Schneider and
Wattenhofer~\cite{SW11} mentioned above asks for a much larger (multiplicative) gap
between $\chi$ and $\Delta$.

\medskip

Theorem~\ref{thm:col} covers in particular the situation
where $\chi(G)\ge \Delta-\sqrt{\Delta}+1$ (and in this case, gives an
efficient algorithm to obtain an optimal coloring of the
graph). Recall that Brooks' theorem (and its algorithmic variants)
colors graphs of maximum degree $\Delta\ge 3$ distinct from
$K_{\Delta+1}$ (or equivalently, with clique number at most $\Delta$) with at most $\Delta$ colors. Our
next result generalizes the algorithmic versions of Brooks' theorem in the following direction.

\begin{thm}\label{thm:1}
  There exists $\Delta_0>0$ such that for every $\Delta\ge \Delta_0$ and
$  2^{59}\log \Delta \le d \le \tfrac{\Delta}{100}$,  there exists a
randomized distributed algorithm that given an $n$-vertex graph of maximum
degree $\Delta$, does the following: either some vertex outputs a clique of size more than
$\Delta-k$ if such a clique exists, or the algorithm finds a coloring with at most
$\Delta-2^{-23}k$ colors. The round complexity is $O(\textsf{T}_{LLL}(n,\poly \,
\Delta)+\textsf{T}_{\de+\Omega(k)}(n,\Delta))$ rounds w.h.p., which is
the minimum of $O(\log_\Delta n+\log_k \Delta)+2^{O(\sqrt{\log \log n})}$ and $2^{O(\log \Delta+\sqrt{\log \log n})}$
  w.h.p., and in particular it is  $O(\log n/\log \log n)$  w.h.p. 
\end{thm}

Here, $\textsf{T}_{\de+\Omega(k)}$ denotes the round complexity of the
$(\de+\Omega(k))$-list coloring problem, which will be defined in
Section~\ref{sec:prel}.

\medskip

We start with some preliminaries on distributed computing, probability, and graph theory in Section~\ref{sec:prel}. We then
prove Theorem~\ref{thm:1} in Section~\ref{sec:thm1}. It turns out that
the proof of Theorem~\ref{thm:1} contains several ingredients that will
be reused in the proof of Theorem~\ref{thm:col}. In
Section~\ref{sec:main}, we prove Theorem~\ref{thm:sharp} and explain
how to adapt the proof of Theorem~\ref{thm:mr1} in~\cite{MR14} to
prove Theorem~\ref{thm:col}. We conclude with some remarks in Section~\ref{sec:conc}.

\section{Preliminaries}\label{sec:prel}

\subsection{Distributed computing}

We consider the classical \textsf{LOCAL} model of computation, which
is a distributed model in which the network corresponds to the graph
under consideration, i.e.\ each vertex of the graph corresponds to a processor, with
infinite computational power, and vertices can communicate with their
neighbors in synchronous rounds (in this model there is no restriction
on the size of the messages exchanged by two neighboring vertices during each round of communication). Each vertex knows the number $n$ of
vertices and its own
id (a distinct integer between $1$ and $n$). In this paper, the
vertices also know the maximum degree $\Delta$ of the graph, and some
number $c$ of colors. Once the
communication between the nodes is over, each vertex outputs a value
(in our case, an integer between $1$ and $c$ corresponding to its
color in a proper coloring of the graph, or some subset of its
neighbors which cannot be colored with $c$ colors). The
complexity of the algorithm is the number of rounds of communication.

\subsection{Vertex coloring}

A \emph{$c$-coloring} of a graph $G$ is an assignment of integers from
$\{1,\ldots,c\}$ to the vertices of $G$ such that any two adjacent
vertices receive distinct colors. The chromatic number $\chi(G)$ of
$G$ is the least $c$ such that $G$ has a $c$-coloring. 

In this paper it will be convenient to consider a slightly more
general scenario, in which the colors available for each vertex are
not necessarily the same. A \emph{list-assignment} $L$ for $G$ is a
collection of lists $L(v)$ of colors, one for each vertex $v$ of
$G$. Given a list-assignment $L$, an \emph{$L$-list-coloring} of $G$
is a coloring of $G$ (i.e.\ any two adjacent vertices receive distinct
colors, as before), with the additional constraint that each vertex
$v$ is colored with a color from its own list $L(v)$.

In the \emph{$(\de+f(\Delta))$-list coloring problem}, the given list-assignment $L$
is such that for each vertex $v$, $|L(v)|\ge d_G(v)+f(\Delta)$ (where
$d_G(v)$ denotes the degree of $v$ in $G$, and $f(\Delta)$ denotes
some function of $\Delta$, the
maximum degree of $G$). When $f(\Delta)\ge 1$, a simple
greedy algorithm shows that $G$ has an
$L$-list-coloring. This is a very useful generalisation of the fact that
any graph of maximum degree $\Delta$ is $(\Delta+1)$-colorable. Let
$\textsf{T}_{\de+f(\Delta)}(n,\Delta)$ be the randomized round complexity of the
$(\de+f(\Delta))$-list coloring problem in $n$-vertex graphs of maximum degree
$\Delta$ in the \textsf{LOCAL} model (i.e. assume that there is exists
a distributed randomized algorithm that solves the
$(\de+f(\Delta))$-list coloring problem in $n$-vertex graphs of maximum degree
$\Delta$ in $\textsf{T}_{\de+f(\Delta)}(n,\Delta)$ w.h.p.).

\smallskip

The following result was proved in~\cite{BEPS12}.

\begin{thm}[\cite{BEPS12}]\label{thm:lcol1}
  $\textsf{T}_{\de+1}(n,\Delta)=O(\log \Delta)+2^{O(\sqrt{\log \log n})}$.
\end{thm}

The following stronger result was then proved in~\cite{ESH15}. 

\begin{thm}[\cite{ESH15}]\label{thm:lcolp}
 For any $\epsilon>0$,  $\textsf{T}_{\de+\epsilon\Delta}(n,\Delta)=O(\log(1/\epsilon))+2^{O(\sqrt{\log \log n})}$.
\end{thm}

Note that Theorem~\ref{thm:lcol1} can be deduced from
Theorem~\ref{thm:lcolp} by simply setting $\epsilon=1/\Delta$.

\smallskip

The setting in which the $(\de+1)$-list coloring and $(\de+\epsilon\Delta)$-list
coloring  problems will be applied is the
following. Let $G$ be a graph of maximum degree $\Delta$ with a subset $S$ of vertices that are colored with
at most $c$ colors. We want to extend the $c$-coloring of $S$ to a
$c$-coloring of $G$ (i.e.\ find a $c$-coloring of $G$ that agrees with
the original coloring on $S$). 

Let $U=V(G)-S$ be the set of uncolored vertices, and for each vertex
$u \in U$, let $L(u)$ be the subset of colors from $1,\ldots,c$ that
do not appear among the neighbors of $u$ in $S$. Note that extending
the $c$-coloring of $S$ to a $c$-coloring of $G$ is the same as finding an $L$-list-coloring
of $G[U]$, the subgraph of $G$ induced by $U$.

Let us denote the degree of a vertex $u\in U$ in
$G[U]$ by $d_U(u)$. The following is a simple, yet very useful observation.

\begin{obs}\label{obs:1}
If $u\in U$ has at least $\ell$ repeated colors in its neighborhood,
then $|L(u)|-d_U(u)\ge c+\ell-d_G(u)\ge c+\ell-\Delta$.
\end{obs}

Hence, extending a $c$-coloring of $S$ to a $c$-coloring of $G$ will
amount to solving a $(\de+f(\Delta))$-list coloring problem in $G[U]$, where
$f(\Delta)$ will depend on the degrees of the vertices of $U$ in $G$
and number of repeated colors in their neighborhoods.

Note that this observation will sometimes be used without an explicit
number of repeated colors (i.e.\ $\ell=0$) and the statement
above simply becomes $|L(u)|-d_U(u)\ge c-d_G(u)$.

\subsection{Probabilistic tools}

Consider a set $X$ of independent random variables, and a set
$B=B_1,\ldots,B_n$ of (typically bad) events, each depending on a subset of the
variables from $X$. Consider the graph $H$ with vertex-set $B$, with
an edge between two events if the set of variables they depend on
intersect. The graph $H$ is called the \emph{event dependency graph}. Let $d\ge 2$ be the maximum degree of $H$, and let $p$ be the
maximum probability of an event from $B$. We say that a family
$\mathcal{B}$ of sets
$B$ as above satisfies
\emph{a polynomial criterion} if there are absolute constants $a,c>0$
such that $apd^c<1$, for each $B\in \mathcal{B}$. The classic
(symmetric version of the) Lov\'asz Local Lemma says that under some
polynomial criterion, we can find a value assignment to the variables of $X$ such
that no event from $B$ holds. Let $\textsf{T}_{LLL}(n,d)$ be the
randomized round
complexity of finding such a value assignment in $H$ in the
\textsf{LOCAL} model of computation (under some polynomial criterion,
as above).

\smallskip

We will use the following result.

\begin{thm}[\cite{CPS14, GHK17}]\label{thm:LLL}
$\textsf{T}_{LLL}(n,d)=\min\{O(\log_d n), 2^{O(\log d+\sqrt{\log \log n})}\}$.
\end{thm}

\begin{proof}
It was proved in~\cite{CPS14} that if $epd^2<1$, there is a distributed randomized algorithm,
running in $H$ in $O(\log_{1/epd^2} (n))$ rounds w.h.p., that finds
a value assignment to the variables of $X$ such
that no event from $B$ holds. Note that under the stricter (but still
polynomial) criterion
$epd^3<1$ we have $1/epd^2\ge  d$ and thus  $O(\log_{1/epd^2}
(n))\le O(\log_d n)$. The second part of the bound,
$\textsf{T}_{LLL}(n,d)=2^{O(\log d+\sqrt{\log \log n})}$, was proved
  in~\cite{GHK17} (under the criterion $2^{15}pd^8<1$).
\end{proof}

It should be noted that in each subsequent application of the
distributed Lov\'asz Local Lemma, the event dependency graph $H$ will only be considered implicitly. The reason is that the variables of $X$ will
be associated to the vertices of some other graph $G$, and the events from $B$ will
correspond to connected subgraphs of $G$ of constant radius. Thus, the
outcome the distributed Lov\'asz Local Lemma will be computed in $G$ directly (the
round complexity is then simply multiplied by a constant, which does not  change
the asymptotic complexity).

\medskip

We shall also use the following version of Talagrand's inequality
(see the appendix in~\cite{MR14}).

\begin{thm}[Talagrand's Inequality]\label{thm:tal}
Let $X$ be a non-negative random variable whose value is determined by $n$ independent trials $T_1$,\ldots,$T_n$ and satisfying the following for some $c$,$r\geq 0$ :
\begin{itemize}
\item changing the outcome of any one trial changes the value of $X$ by at most $c$.
\item for any $s$, if $X\geq s$ then there is a set of at most $rs$
  trials whose outcomes certify $X\geq s$.
\end{itemize}
Then for any $t \ge 0$, 
$$\mathbb{P}\left(\lvert X - \mathbb{E}(X) \rvert > t + 20c\sqrt{r\mathbb{E}(X)}+64c^2r\right) \leq 4\cdot \exp \left(-\frac{t^2}{8c^2r(\mathbb{E}(X)+t)} \right)$$
\end{thm}

\subsection{The dense decomposition}

The graph decomposition described in this section is due to Reed~\cite{Ree98} (see
also~\cite{MR02,MR14}). A somewhat similar (although not completely equivalent)
decomposition was recently used by Harris, Schneider, and Su~\cite{HSS16}
(see also~\cite{CLP18}) in the context of distributed $(\Delta+1)$-coloring
algorithms.

Consider a graph $G=(V,E)$ of maximum degree $\Delta$. We call a vertex \textit{$d$-dense} if its neighborhood has more than
$\binom{\Delta}{2}-d\Delta$ edges (note that $d$ might depend on
$\Delta$). A vertex $v$ that is not $d$-dense is said to be
\emph{$d$-sparse}. 

\medskip

We say that $S,X_1,X_2,\ldots,X_t$ is a \emph{$d$-dense decomposition
  of $G$} if each of the following holds:
\begin{enumerate}
\item $S,X_1,X_2,\ldots,X_t$ partition $V$;
\item every $X_i$ has between $\Delta - 8d$ and $\Delta+4d$ vertices;
\item there are at most $8d\Delta$ edges between $X_i$ and $V-X_i$;
\item a vertex is adjacent to at least $\frac{3\Delta}{4}$ vertices of
  $X_i$ if and only if it is in $X_i$;
\item every vertex in $S$ is $d$-sparse.
\end{enumerate}

The sets $X_i$ are called the \emph{dense components} and $S$ is
called the \emph{sparse component}.
Note that a simple consequence of (4) and (2) is that each dense component has diameter at
most 2, provided that $d\le \tfrac{\Delta}{8}$.

\begin{lem}\label{lem:decompo}
A $d$-dense decomposition of $G$ can be constructed in $O(1)$ rounds for every $d\leq \frac{\Delta}{100}$.
\end{lem}

\begin{proof}
Each $d$-dense vertex $v$ learns its neighborhood at distance 2 and applies the following procedure in parallel
in order to build a cluster $D_v$:

\smallskip

\noindent {\bf Phase 1.}

\begin{itemize}
\item[\textbf{1. }] $D_v = v\cup N(v)$
\item[\textbf{2. }] While there is some vertex $u$ in $D_v$ with $|N(u)\cap D_v| < \frac{3\Delta}{4}$, remove $u$ from $D_v$.
\item[\textbf{3. }] While there is some vertex $u$ outside $D_v$ with
  $|N(u)\cap D_v| \geq \frac{3\Delta}{4}$, add $u$ to $D_v$.
\end{itemize}

Note that only vertices at distance at most two from $v$ are added or
removed from $D_v$, so this 3-step procedure can indeed be performed in $O(1)$ rounds. It
follows from Lemma 15.2
in~\cite{MR02} that (i) $v\in D_v$, (ii) every vertex $x$ is in $D_v$ if and only if
$|N(x)\cap D_v|\geq \frac{3\Delta}{4}$, (iii) $\Delta -8d \leq |D_v|
\leq \Delta+4d$, (iv) there are at most $8d\Delta$ edges between $D_v$ and $V-D_v$, and (v) if $x$ and $y$ are two $d$-dense vertices and
$D_x\cap D_y\neq \emptyset$ then $x\in D_y$ and $y \in D_x$.

\smallskip

\noindent{\bf Phase 2.} Now, every $d$-sparse vertex that is not in any cluster $D_v$ \emph{joins} the
set $S$, and every other (sparse or dense) vertex $v$ considers the $d$-dense vertex $u$
with smallest id such that $v\in D_u$, and \emph{joins} $D_u$ (while leaving
all the other sets $D_w$ it was part of).
After this step, each dense vertex $v$ sends to its neighbors at
distance at most two the id of the set $D_u$
it joined, and each sparse vertex $v$ checks whether the set $D_u$
it joined during Phase 2 is such that $u$ also joined $D_u$ during
Phase 2. If this is the case $v$ remains
in $D_u$, and if not $v$ joins $S$.

\smallskip

We now prove that $S$ together with the resulting non-empty clusters $D_v$ form
the desired $d$-dense decomposition of $G$.
We first note that these sets partition $V$, as each vertex joining
a cluster $D_u$ also leaves all the other clusters it was part
of. Observe now that property (v) above implies that if some $d$-dense
vertex $v$ does not join $D_v$ during Phase 2, then no $d$-dense
vertex of $D_v$ joins $D_v$, and since the $d$-sparse vertices of
$D_v$ join $S$, then $D_v$ is empty after Phase 2. On the other hand,
property (v) implies that if $v$ joined $D_v$ during Phase 2, then all
vertices of $D_v$ (sparse or dense) also join $D_v$ during this
phase. Using properties (i)--(v) above, this concludes the proof of Lemma~\ref{lem:decompo}. 
\end{proof}

\section{Graphs with small clique number}\label{sec:thm1}

In this section we prove Theorem~\ref{thm:1}. The proof is a simple combination of ideas developed in the proofs of
Lemmas 10 and 16 in~\cite{MR14} (see also Section 10.3
in~\cite{MR02}). The proofs
there are given for a slightly different range of parameters, so we
decided to
include the full proof here instead of simply pointing to appropriate
parts of their results. More specifically, we will need the following
two results.

\begin{lem}\label{lem:sparse}
Let $G$ be a graph of (sufficiently large) maximum degree $\Delta$ and let $\ell\ge 2^{54}\log
\Delta$. Then there is a distributed randomized algorithm that finds a
partial coloring of $G$ with $\Delta/2$ colors in
$\textsf{T}_{LLL}(n,\poly \Delta)$ rounds w.h.p., such that for each
uncolored vertex $v$ with at least $\ell\Delta$ pairs of non-adjacent
vertices in $N(v)$, there are more than $2^{-18}\ell$ repeated colors in $N(v)$.
\end{lem}

\begin{lem}\label{lem:dense}
Let $S,X_1,\ldots,X_t$ be a $2^{-4}k$-dense decomposition of a graph $G$ of
maximum degree $\Delta\ge 30k$ and clique number at most
$\Delta-k$. Then there is a distributed randomized algorithm that extends any $c$-coloring of $S$ with $c\ge\Delta-k/48$ colors to a $c$-coloring of $G$  in
$O(\textsf{T}_{\de+\Omega(k)}(n,\Delta))$ rounds, w.h.p.
\end{lem}

We now explain how these two results can be combined to provide a
proof of Theorem~\ref{thm:1}. It should be mentioned that we have made
no significant effort to optimize the various constants appearing
throughout the proof, and have chosen instead to focus on making
the proof as simple as
possible. Lemmas~\ref{lem:sparse} and~\ref{lem:dense} will
be proved at the end of the section.

\bigskip

\noindent \emph{Proof of Theorem~\ref{thm:1}.}
If $G$ contains a clique on more than $\Delta-k$ vertices,
it can be found in $O(1)$ rounds so we may assume in the remainder that
$G$ has clique number at most $\Delta-k$.

We start by using Lemma~\ref{lem:decompo} to compute a $2^{-4}k$-dense
decomposition $S,X_1,X_2,\ldots,X_t$ of $G$ (note that we have
$2^{-4}k\le 2^{-4}\Delta/30\le \Delta/100$, as required). Let $T$ be the vertices
of $S$ with degree at least $\Delta-2^{-5}k$ in $S$. Since each vertex of
$v\in T$ is $2^{-4}k$-sparse, $N(v)$ contains at least $${\Delta-2^{-5}k\choose
2}-{\Delta\choose
2}+2^{-4}k\Delta\ge 2^{-5}k\Delta$$ pairs of
non-adjacent vertices in $S$.

Using Lemma~\ref{lem:sparse} with $\ell=2^{-5}k$, we then obtain a partial coloring of $S$
with at most $\Delta/2\le \Delta-2^{-24}k$ colors in $\textsf{T}_{LLL}(n,\poly \Delta)$ rounds w.h.p.,
such that each uncolored vertex of $T$ has more than $2^{-23}k$ repeated colors
in its neighborhood. Let $U$ be the set of uncolored vertices of $S$,
and for each vertex of $v\in U$, let
$L(v)$ be the set of colors from $1,\ldots,\Delta-2^{-24}k$ that do not appear in the neighborhood of
$v$. We claim that 
\begin{longequation}\label{eq:1}
for each $v\in U$, $|L(v)|-d_U(v)\ge 2^{-24}k$,
\end{longequation}
where
$d_U(v)$ denotes the number of neighbors of $v$ in $U$, or
equivalently the degree of $v$ in $G[U]$. 

\smallskip

To see why (\ref{eq:1}) holds,
consider first the case $v\in U-T$. Observe that in this case $v$ has
degree at most
$\Delta-2^{-5}k$ in $S$, and thus (\ref{eq:1}) follows directly from
Observation~\ref{obs:1} with $c=\Delta-2^{-24}k$, $\ell=0$, and
$d_S(v)\le\Delta-2^{-5}k$ (which implies $c-d_S(v)\ge\Delta-2^{-24}k-\Delta+2^{-5}k\ge 2^{-24}k$).

Assume now that $v\in
U\cap T$. Since each uncolored vertex
of $T$ has more than $2^{-23}d$ repeated colors
in its neighborhood, (\ref{eq:1}) follows directly from
Observation~\ref{obs:1} with $c=\Delta-2^{-24}k$ and $\ell=2^{-23}k$
(which implies $c-\Delta+\ell=\Delta-2^{-24}k-\Delta+2^{-23}k=2^{-24}k$). This concludes the proof of (\ref{eq:1}).

\medskip

It
follows from (\ref{eq:1}) (and the discussion before Observation~\ref{obs:1}) that we can extend the partial coloring of $S$ to
all the vertices of $S$ in $\textsf{T}_{\de+\Omega(k)}(n,\Delta)$ rounds, w.h.p.

\medskip

It remains to extend the coloring of $S$ to the dense components
$X_1,\ldots,X_t$. Using Lemma~\ref{lem:dense}, the coloring of $S$ can
then be extended to $X_1,\ldots,X_t$ in
$O(\textsf{T}_{\de+\Omega(k)}(n,\Delta))$ rounds, w.h.p. It follows that the round complexity is $O(\textsf{T}_{LLL}(n,\poly \,
\Delta)+\textsf{T}_{\de+\Omega(k)}(n,\Delta))$ rounds w.h.p. Using
Theorems~\ref{thm:LLL} and~\ref{thm:lcolp}, this is
the minimum of $O(\log_\Delta n+\log_k \Delta)+2^{O(\sqrt{\log \log n})}$ and $2^{O(\log \Delta+\sqrt{\log \log n})}$
  w.h.p., and in particular it is  $O(\log n/\log \log n)$  w.h.p., for any value of $\Delta$, which concludes the proof of
Theorem~\ref{thm:1}.\hfill $\Box$

\medskip

We now turn to the proof of Lemma~\ref{lem:sparse}, which is a
classical application of the probabilistic method, see Lemma 10
in~\cite{MR14}, or Section 10.3
in~\cite{MR02} (which considered a slightly smaller range of values for the
parameter $d$, namely $d=\Omega(\log^3 \Delta)$ instead of $d=\Omega(\log\Delta)$).

\medskip

\noindent {\it Proof of Lemma~\ref{lem:sparse}.}
Let $C=\tfrac{\Delta}2$. We apply the following simple randomized procedure
in two steps (see Section 3.2 in~\cite{MR14} or Section 10.3 in~\cite{MR02}).
\begin{itemize}
\item[\textbf{1. }] Each vertex $v$ with at least $\ell\Delta$ pairs of non-adjacent
vertices in $N(v)$ chooses a color uniformly at
  random from $\{1,\ldots,C\}$, independently of the other vertices.
\item[\textbf{2. }] If the color chosen by $v$ is also chosen by a
  neighbor of $v$ at Step 1, then $v$ uncolors itself.
\end{itemize}

Note that two adjacent vertices that received the same color at Step 1
will both be uncolored at Step 2 (it is the reason why this procedure is
sometimes called the \emph{wasteful coloring procedure}).

\medskip

The classical analysis of the procedure is as follows. For $v$ with at least $\ell\Delta$ pairs of non-adjacent
vertices in $N(v)$, we define $B_v$ as the event that there are at most
$2^{-18}\ell$ repeated colors in $N(v)$. We will prove that
$\Pr(B_v)\le \Delta^{-2^6}$. Since each event $B_v$ only depends of
the colors of the vertices at distance at most two from $v$, the
maximum degree of the event dependency graph associated to the events $B_v$ is at
most $\Delta^4$ and it follows that we can find a
partial color assignment avoiding all events $B_v$ in
$\textsf{T}_{LLL}(n,\Delta^4)$ rounds w.h.p., as $2^{15}\Delta^{-2^6} (\Delta^4)^{8}\le
\tfrac1{\Delta}<1$ for sufficiently large $\Delta$.

Let $P_v$ be number of pairs of non-adjacent vertices $u,w$ in $N(v)$
such that (1) $u$ and $w$ were assigned the same color and (2) no
other neighbor of $u$, $v$, or $w$ was assigned the same color. Note
that the
number of repeated colors in $N(v)$ after Step 2 is at least
$P_v$. Since $v$ has at least $\ell\Delta$ pairs of non-adjacent vertices
in $N(v)$ and at most $3\Delta$ vertices are neighbors of $u$, $v$, or
$w$, we have $\Exp(P_v)\ge \ell\Delta \cdot \tfrac1{C} \cdot
(1-\tfrac1{C})^{3\Delta}\ge 2\ell e^{-12}\ge 2^{-17}\ell$, using that $C=\Delta/2$ and $1-x\ge \exp(-2x)$ for any
$0\le x\le \tfrac12$.

It follows that if $B_v$ holds, then $|P_v-\Exp(P_v)|\ge 2^{-18}\ell$. We
now prove using Talagrand's Inequality (Theorem~\ref{thm:tal}), that
$P_v$ is highly concentrated, and thus the probability that it differs
from its expectation by at least $2^{-18}\ell$ is small. To this end,
define $Y_v$ as the number of colors assigned to at least one pair
of non-adjacent neighbors of $v$, and let $Z_v$ be the number of
colors assigned to at least one pair
of non-adjacent neighbors of $v$ (call them $u$ and $w$), and also to
at least one distinct neighbor of $u$, $v$, or $w$. We clearly have
$P_v=Y_v-Z_v$, and thus, if $P_v$ differs from its expectation by at
least $2^{-18}\ell$, then $Y_v$ or $Z_v$ differs from its
expectation by at least $2^{-19}\ell$.

\smallskip

Note first that $\Exp(Z_v)\le \Exp(Y_v)\le C \ell\Delta \cdot
\tfrac1{C^2}=2\ell$. Observe also that any change on the color of
a single vertex affects
the values of
$Y_v$ and $Z_v$ by at most 1 (removing the old color can only decrease the
variables, by at most 1, and adding the new color can only increase
the variables, also by at most 1). Moreover, if
$Y_v\ge s$ there is a set of $2s$ color assignments to the vertices of
$N(v)$ that certify this, and if $Z_v\ge s$ there is a set of $3s$
color assignments to the vertices at distance at most two from $v$
that certify this. 
We can thus apply Theorem~\ref{thm:tal} to the variable
$Y_v$ with $c=1$ and $r=2$, and to the variable $Z_v$ with $c=1$ and
$r=3$. We obtain

$$\mathbb{P}\left(\lvert Y_v - \mathbb{E}(Y_v) \rvert > t + 20\sqrt{2\mathbb{E}(Y_v)}+2^7\right) \leq 4\cdot \exp \left(-\frac{t^2}{2^4(\mathbb{E}(Y_v)+t)} \right)$$

Take $t=2^{-19}\ell-20\sqrt{2\mathbb{E}(Y_v)}-2^7$ and note that $t\ge 2^{-19}\ell
-2^6\sqrt{\ell}-2^7\ge 2^{-20}\ell$ for sufficiently large $\ell$ (recall that
$\ell=\Omega(\log \Delta)$ and $\Delta$ is assumed to be sufficiently
large). Note also that $\Exp(Y_v)+t\le 2\ell+2^{-19}\ell \le 4\ell$. As a consequence,

$$\mathbb{P}\left(\lvert Y_v - \mathbb{E}(Y_v) \rvert > 2^{-19}\ell
\right) \leq 4\cdot \exp \left(-\frac{2^{-40}\ell^2}{2^6\ell} \right)\le 4\cdot \exp \left(-2^{-46}\ell \right).$$

For $Z_v$ we obtain similarly: $$\mathbb{P}\left(\lvert Z_v -
  \mathbb{E}(Z_v) \rvert > t + 20\sqrt{3\mathbb{E}(Z_v)}+3\cdot
  2^6\right) \leq 4\cdot \exp \left(-\frac{t^2}{24(\mathbb{E}(Z_v)+t)}
\right).$$

By taking $t=2^{-19}\ell-20\sqrt{3\mathbb{E}(Z_v)}+3\cdot
  2^6$, and noting that for sufficiently large $\Delta$, we have $2^{-20}\ell\le t\le
  2^{-19}\ell$ and $\Exp(Z_v)+t\le 4\ell$, we obtain:

$$\mathbb{P}\left(\lvert Z_v - \mathbb{E}(Z_v) \rvert > 2^{-19}\ell
\right) \leq 4\cdot \exp \left(-\frac{2^{-40}\ell^2}{2^7\ell} \right)\le
4\cdot \exp \left(-2^{-47}\ell \right).$$

Note that since $\ell\ge 2^{54}\log \Delta$, we have $4\cdot \exp
\left(-2^{-47}\ell \right)\le \tfrac12 \Delta^{-2^6}$ for sufficiently
large $\Delta$. It follows that the probability that $P_v$ differs
from its expectation by at least $2^{-18}\ell$ is at most
$\Delta^{-2^6}$, as desired. This concludes the proof of Lemma~\ref{lem:sparse}. \hfill $\Box$

\medskip

We conclude this section with the proof of Lemma~\ref{lem:dense},
which is an extension of the proof of Lemma 16 in~\cite{MR14}
(which only considered the special case $k=\sqrt{\Delta}$). A
significant difference is that in~\cite{MR14}, the coloring is
extended to each dense set sequentially, while here we color all the
dense sets $X_i$ at once.

\medskip

\noindent \emph{Proof of Lemma~\ref{lem:dense}.}
By the definition of a $2^{-4}k$-dense decomposition, recall that 

\begin{enumerate}
\item $X_i$ has between $\Delta - k/2$ and $\Delta+k/4$ vertices.
\item There are at most $k\Delta/2$ edges between $X_i$ and $V-X_i$.
\item a vertex is adjacent to at least $\frac{3\Delta}{4}$ vertices of $X_i$ if and only if it is in $X_i$.
\end{enumerate}
Consider a maximal matching in the complement of $X_i$ (the graph with
vertex-set $X_i$ in which two vertices are adjacent if and only if
they are non-adjacent in $G$). Note that the set $C$ of vertices of $X_i$ not
covered by the matching forms a clique (of size at most $\Delta-k$),
and since $X_i$ has size at least $\Delta-k/2$, the matching has
size at least $\tfrac12 (\Delta-k/2-|C|)\ge \tfrac12
(\Delta-k/2-\Delta+k)\ge k/4$. 

Let $M_i$ be a set of precisely
$k/4$ pairs of distinct vertices $(u_1,v_1),\ldots,(u_{k/4},v_{k/4})$
of $X_i$, that are
pairwise disjoint, and such that for any $1\le j\le k/4$, $u_j$ is
non-adjacent to $v_j$ in $G$. Let $U_i$ be the set of
vertices of $X_i$ not covered by $M_i$, and note that $\Delta-k\le |U_i|
\le \Delta-\tfrac{k}4$. We
say that a vertex $w\in U_i$ dominates a pair $(u_j,v_j)$ of $M_i$ if $w$ is
adjacent to both $u_j$ and $v_j$.
Fix a pair $(u_j,v_j)$ in $M_i$, and observe that by
property (3) above, the
number of vertices of $U_i$ that dominate $(u_j,v_j)$ is at
least $3\Delta/2-4|M_i|-|U_i|\ge 3\Delta/2-k-\Delta +\tfrac{k}4\ge \Delta/2-\tfrac{3k}4\ge
|U_i|/3$ (the final inequality follows from the fact that $k\le
\Delta/30\le 2\Delta/9$). A simple double counting argument then
shows that the number of vertices of $U_i$ that dominate at least
$|M_i|/6$ pairs of $M_i$ is at least $|U_i|/5$. Let
$Z_i$ be the set of such vertices of $U_i$, and note that $|Z_i|\ge |U_i|/5\ge
\Delta/5-k/5\ge \Delta/6$ (since $k\le \Delta/30\le \Delta/6$).

We now divide $U_i-Z_i$ into 3
parts: the set $W^0_i$ of vertices of $U_i-Z_i$ of degree at most $\Delta-k$
in $G$, the set $W^+_i$ of vertices of $U_i-(Z_i\cup W^0_i)$ with at least
$|Z_i|/4$ neighbors in $Z_i$, and the set $W_i^-$ of vertices of $U_i-(Z_i\cup
W^0_i)$ with less than
$|Z_i|/4$ neighbors in $Z_i$. Note that the vertices of $W^-_i$ have degree
at least $\Delta-k$ in $G$ (since they are not in $W_i^0$) and thus they
have at least $\Delta-k-(|X_i|-3|Z_i|/4)\ge
\Delta-k-\Delta-k/4+\Delta/8\ge \Delta/8-5k/4\ge \Delta/12$
neighbors outside of $X_i$ (the final inequality follows from the
fact that $k\le \Delta/30$). Since there are at most $k\Delta/2$ edges
between $X_i$ and $v-X_i$, we have $|W^-_i|\cdot\Delta/12 \le k\Delta/2$ and
thus $|W^-_i|\le 6k$.

We are now ready to extend the coloring of $S$ to the sets $X_i$. We
proceed in the following order.

\begin{enumerate}
\item We start by coloring the vertices covered by the $M_i$'s. Consider the graph $H_1$ obtained from $G$ by identifying the vertex $u_j$
  with the vertex $v_j$, for each pair $(u_j,v_j)$ of each set
  $M_i$. The coloring of $S$ in $G$ corresponds to a coloring of $S$
  in $H_1$, and we want to extend this coloring of $S$ in $H_1$ to
  the newly created vertices (in $G$, this will correspond to an
  extension of the coloring of $S$ to all the vertices covered by the
  $M_i$'s, such that in any pair $(u_j,v_j)$ of some $M_i$, the two
  vertices $u_j$ and $v_j$ are
  assigned the same color). 

Note that each newly created vertex $x$
  in some $M_i$ has at
  most $\frac{\Delta}{4}+\frac{\Delta}{4}=\Delta/2$ neighbors outside
  $X_i$ and at most $|M_i|\le k/4$ neighbors among the newly created
  vertices of $M_i$, thus $x$ has degree at most 
  $\Delta/2+k/4$ in $H_1$. If $L(x)$ denotes the list of colors
  available for $x$ in $H_1$ (i.e.\ the colors that do not appear among
  the neighbors of $x$ in $S$), then it follows from
  Observation~\ref{obs:1} with $c\ge \Delta-k/48$, $\ell=0$, and
  $d_{H_1}(x)\le\Delta/2+k/4$ that $|L(x)|$ exceeds the number of neighbors
  of $x$ in $H_1-S$ by at least $c-d_{H_1}(x)\ge \Delta-k/48-\Delta/2-k/4\ge
  \Delta/48$. We can thus extend the coloring of $S$ to the
  newly created vertices of $H_1$ in
  $\textsf{T}_{\de+\Omega(\Delta)}(n,\Delta)$ rounds, w.h.p. In $G$, this corresponds to a coloring of the
vertices covered by the $M_i$'s extending the coloring of $S$, such that for any pair $(u,v)$ in any $M_i$, $u$ and $v$ have the same
color.

\item We then color $W^-=\bigcup_i W_i^-$. To do this, observe that since each set $W_i^-$ has
  size at most $6k$, and each corresponding set $M_i$ has size $k/4$,
  it follows that each vertex $v\in W_i^-$ has at most
  $\Delta/4+6k+2\cdot k/4$ neighbors that are either in $W^-$ or
  already colored. Combining this with Observation~\ref{obs:1} (with $c\ge\Delta-k/48$ and
  $\ell=0$),
  we can then extend the current coloring to $W^-$ in
  $\textsf{T}_{\de+\Omega(\Delta)}(n,\Delta)$ rounds w.h.p.

\item We then color $W^+=\bigcup_i W_i^+$. These vertices have at least $|Z_i|/4\ge
  \Delta/24$ neighbors in the corresponding set $Z_i$, which are all
  uncolored at this point (they will be colored at the next step),
  thus each vertex of $W^+$ has at most $23\Delta/24$ neighbors that
  are either in $W^+$ or already colored. Combining this with
  Observation~\ref{obs:1} (with $c\ge \Delta-k/48$ and
  $\ell=0$), we can then extend the current coloring to $W^+$ in
  $\textsf{T}_{\de+\Omega(\Delta)}(n,\Delta)$ rounds, w.h.p.
\item We now color $Z=\bigcup_i Z_i$. Since each vertex of some $Z_i$ is adjacent to both
  members of at least $\frac{|M_i|}{6}$ pairs of $M_i$, it has at least
  $\frac{|M_i|}{6}=k/24$ repeated colors in its
  neighborhood. Combining this with Observation~\ref{obs:1} (with $c\ge\Delta-k/48$ and
  $\ell=k/24$, and thus $c+\ell-\Delta=k/48$),
  we can then extend the current coloring to $W^-$in
  $\textsf{T}_{\de+\Omega(k)}(n,\Delta) $ rounds, w.h.p.
\item We now color $W^0=\bigcup_i W^0_i$. Each vertex in this set has degree
  at most $\Delta-k$ in $G$ and can thus using Observation~\ref{obs:1},
  we can then extend the current coloring to $W^+$  in
  $\textsf{T}_{\de+\Omega(k)}(n,\Delta)$ rounds, w.h.p.
\end{enumerate}

This concludes the proof of Lemma~\ref{lem:dense}.
\hfill $\Box$

\section{Graphs with chromatic number close to the maximum degree}\label{sec:main}


In this section, we prove the main result of this paper.

\smallskip

We start with the (fairly simple) proof of Theorem~\ref{thm:sharp}, and
then prove Theorem~\ref{thm:col}, or rather explain how it can be
deduced from appropriate parts of the proof of Theorem~\ref{thm:mr1}
in~\cite{MR14}. It should be noted that our assumption that $c\ge
\Delta-k_\Delta+1$ makes the proof of Theorem~\ref{thm:col}
significantly easier than the proof of Theorem~\ref{thm:mr1}
in~\cite{MR14}, where the main difficulty comes from the case $c=\Delta-k_\Delta$.

\subsection{Reducers}\label{sec:reducer}

A \emph{stable set}, or \emph{independent set}, is a set of pairwise
non-adjacent vertices.
A \emph{$c$-reducer} in a graph $G$ is a subset $D$ of vertices consisting of a clique $C$ with $c - 1$
vertices and a disjoint stable set $S$
such that every vertex of $C$ is adjacent to all of $S$ but none of
$V(G)-D$ (see Figure~\ref{fig:ope}, right). Given a graph $G$ with a $c$-reducer $D=(C,S)$, the graph
$H$ obtained from $G$ by removing $C$ and identifying all the vertices
of $S$ into a single vertex is called the \emph{reduction} of $G$ with
respect to $D$ (see Figure~\ref{fig:ope}, left). Note that
$G$ is $c$-colorable if and only if $H$ is $c$-colorable, and thus
$c$-reductions preserve $c$-colorability and non-$c$-colorability.

\medskip

\noindent {\it Proof of Theorem~\ref{thm:sharp}.}
Let $\Delta$ be an integer, and assume that either 

\begin{itemize}
\item $c\le
\Delta-k_\Delta-1$, or 
\item  $c=
\Delta-k_\Delta$ and $\Delta=(k_\Delta+1)(k_\Delta+2)$.
\end{itemize}

For $i\ge 1$, we define a graph $G_i$ of maximum degree $\Delta$ and a
subset $C_i$ of $G_i$ inductively as follows. $G_1$ is the
complete graph on $c+1$ vertices, and $C_1$ is the set of vertices of $G_1$. For any $i \ge 2$,  $G_{i}$ is obtained from $G_{i-1}$ by
removing an arbitrary vertex $v_{i-1}$ of $C_{i-1}$, adding a stable set
$S_i$ of size $\Delta-c+2$ and a $(c-1)$-clique $C_i$ such that 
(1) each neighbor of $v_{i-1}$ in $G_{i-1}$ is adjacent to exactly one vertex
of $S_i$, and (2)
each vertex of $S_i$ is adjacent to all the vertices of $C_i$. The
construction of $G_i$ from $G_{i-1}$ is depicted in
Figure~\ref{fig:ope}.

\begin{figure}[htbp]
\begin{center}
\includegraphics[width=5cm]{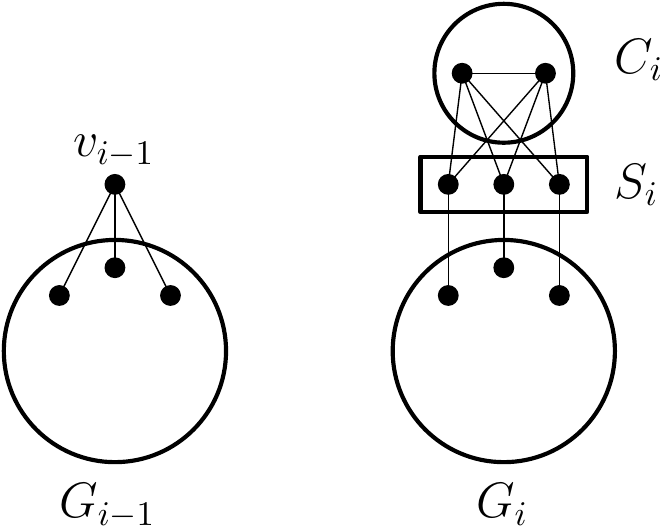}
\caption{The construction of the graph $G_i$ from $G_{i-1}$.}
\label{fig:ope}
\end{center}
\end{figure}

In order to make sure that the maximum degree of $G_i$ is at most
$\Delta$, while performing (1) we split as evenly as possible the degree of $v_{i-1}$ between
the vertices of $S_i$ (each edge between $v_{i-1}$ and some neighbor $u$ in
$G_{i-1}$ becomes an edge joining $u$ and some vertex of $S_i$ in $G_i$,
and we want the degrees of the vertices of $S$ to be as balanced as possible). Since
$|S_i|= \Delta-c+2$, each vertex of $C_i$ has degree
$\Delta$ in $G_i$. 
Each vertex of $S_i$
must also have
degree at most $\Delta$ so it can have up to $\Delta-c+1$ neighbors in
$G_{i-1}$. Since $v_{i-1}$ has degree at most $\Delta$, and
$(\Delta-c+2)(\Delta-c+1)\ge  \Delta$, the edges incident to $v_{i-1}$ in
$G_{i-1}$ can be split among the vertices of $S_i$ in such way that each vertex of $S_i$
has degree at most $\Delta$ in $G_i$.

\begin{figure}[htbp]
\begin{center}
\includegraphics[width=10cm]{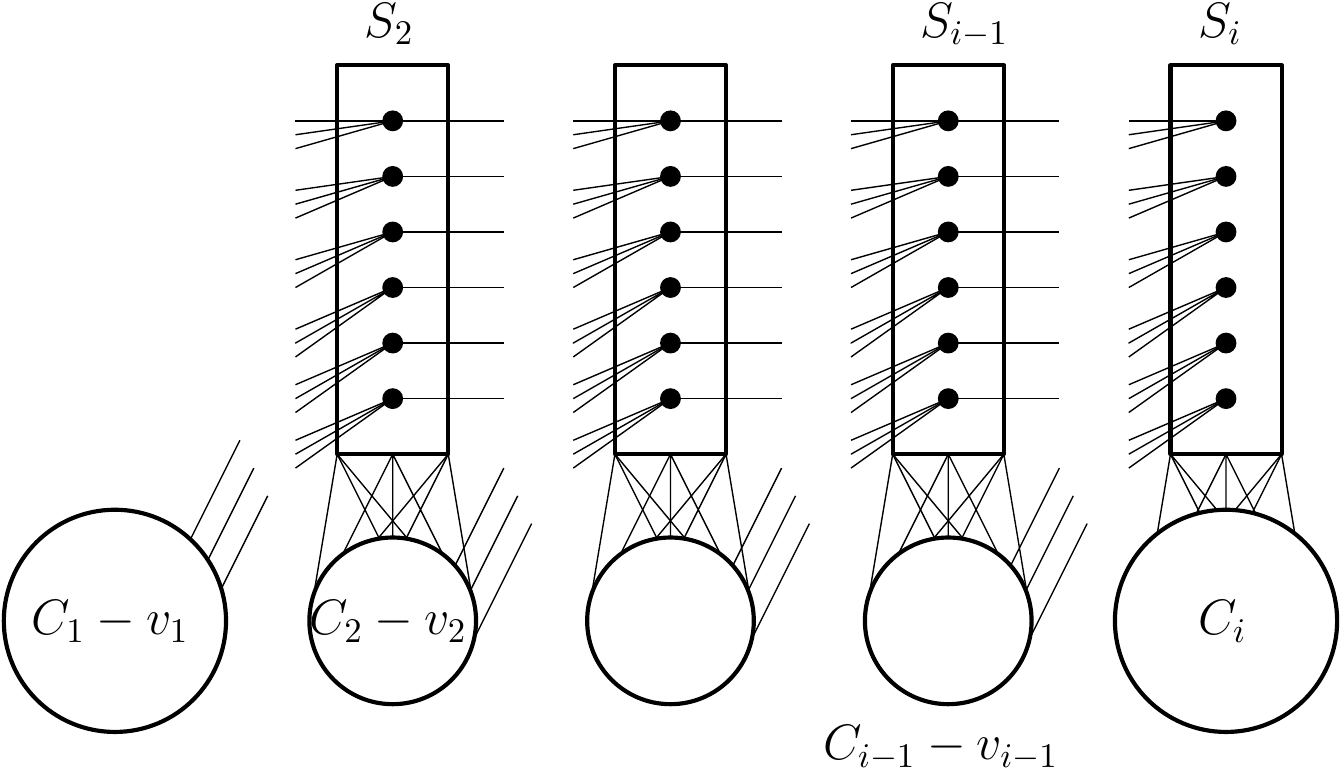}
\caption{The graph $G_i$. Cliques are represented by circles and
  stable sets by rectangles.}
\label{fig:band}
\end{center}
\end{figure}

We now make a couple of remarks on $G_i$. It can be observed that
$G_{i-1}$ is the reduction of $G_i$ with respect to some $c$-reducer,
and since $G_1$ is a clique on $c+1$ vertices and reductions preserve
$c$-non-colorability, $G_i$ is not
$c$-colorable. It is also easy to see that any proper subgraph of
$G_i$ has
chromatic number at most $c$ (see Observation 3 in~\cite{MR14}).
Note that $G_i$ consists of $i$ layers, each being the union of a
clique of size at most $c\le \Delta$ and a
stable set of size at most $\Delta-c+2\le \Delta$ (see Figure~\ref{fig:band}),
and thus $G_i$ has diameter at least $\tfrac{n}{2\Delta}$, where $n$
denotes the number of vertices of $G_i$. Let $G$ be
the graph obtained from $G_i$ by deleting a single edge between a
vertex of layer $i/2$ (i.e.\ a vertex that was added at step $i/2$) and
a vertex of layer $i/2+1$. As a proper subgraph of $G_i$, $G$ has
maximum degree at most $\Delta$ and 
chromatic number at most $c$, and it can be checked that any ball of
radius less than $\tfrac{n}{8\Delta}$ in $G_i$ is isomorphic to a ball
of the same radius in $G$ (by performing each step of the
construction in precisely the same way we can make sure that all the graphs
induced by any fixed number of consecutive layers except the first and
last ones are isomorphic).
Since $G_i$ is not $c$-colorable, it follows from a classical observation\footnote{This observation is
not explicitly stated in~\cite{Lin92}, but is the essence of the
proof of Theorem 3.1 in that paper. Namely, if two graphs $G,H$ are
such that $|V(H)|\le |V(G)|$ and any ball of radius $t+1$ in $H$ is isomorphic to some ball
of radius $t+1$ in $G$, then the $t$-neighborhood
graph $N_t(H)$ of $H$ is a subgraph of the $t$-neighborhood
graph $N_t(G)$ of
$G$. Since $H$ is a subgraph of $N_t(H)$, it is also a subgraph of
$N_t(G)$. It follows from ~\cite[Proposition 2.3(1)]{Lin92} that a graph $G$ can be
$c$-colored in $t$ rounds if and only if its $t$-neighborhood graph
$N_t(G)$ is $c$-colorable. This
implies that $G$ cannot be colored with less than $\chi(H)$ colors in
$t$ rounds in the \textsf{LOCAL} model.} 
of Linial~\cite{Lin92}, that $G$ cannot be colored optimally
(i.e.\ with $c$ colors)  in less than $\tfrac{n}{8\Delta}$
rounds.
This concludes the proof of Theorem~\ref{thm:sharp}.
\hfill $\Box$

\subsection{Overview of the proof of Theorem~\ref{thm:col}}

We start by considering the first part of the statement of Theorem~\ref{thm:col}: if
$G$ is not $c$-colorable, then some vertex is supposed to output  a \emph{certificate} that $G$
is not $c$-colorable. In order to do so, we will use the following result of Molloy and Reed (Theorem 5
in~\cite{MR14}).

\begin{thm}\label{thm:mrlocal}
For sufficiently large $\Delta$, and for $c \ge \Delta -k_\Delta+1$,
if $G$ has maximum degree at most $\Delta$, and $\chi(G) > c $,
then there is some vertex $v$ in $G$ such that the subgraph induced by
$\{v\} \cup N(v)$ is not $c$-colorable.
\end{thm}

\medskip

In the \textsf{LOCAL} model of computation, testing the
$c$-colorability of all closed neighborhoods (i.e.\ all the balls of radius 1) in $G$ can be done in a
constant number of rounds, and any vertex finding a non $c$-colorable
subgraph in its closed neighborhood can simply output this subgraph as
a certificate of non $c$-colorability of $G$. It might be worth
pointing that we heavily use the unbounded computational power of the
nodes (and the unbounded bandwidth of the edges) in the \textsf{LOCAL} model here when $\Delta \gg \log n$. However, when $\Delta=O(\log n)$, all the closed
neighborhoods have logarithmic size, so testing their $c$-colorability
takes polynomial time (in $n$) in any classical model of
computation. Moreover, when $\Delta=O(1)$ the same task can be
performed in constant time in any classical model of
computation.

\medskip

We can now assume that $G$ is $c$-colorable, and the goal is to
find a $c$-coloring of $G$ in $O(\textsf{T}_{\de+1}(n,\Delta) )+O((\log\Delta)^{13/12})\cdot \textsf{T}_{LLL}(n,\poly \,
\Delta)$
rounds w.h.p. The high-level description of the proof is as follows: we
set $d=10^6\sqrt{\Delta}$ and
start by computing a $d$-dense decomposition $S,X_1,\ldots,X_t$ of
$G$. We then delete all the sets $X_i$ that are $c$-reducers or such that $\overline{G[X_i]}$
has a matching of size at least $100\sqrt{\Delta}$. These sets will be
colored at the very end, once the rest of the graph will be
colored, using a proof very similar to that of Lemma~\ref{lem:dense},
in $O(\textsf{T}_{\de+1}(n,\Delta))$ additional rounds (Lemmas~\ref{lem:8} and~\ref{lem:16}). So we can assume
that no set $X_i$ is a $c$-reducer or has a large antimatching. Using this assumption, we
then find a specific $c$-coloring in each set $X_i$, independently of the
other sets $X_j$, with desirable properties (Lemma~\ref{lem:25}). Using this coloring
of each set $X_i$, we will construct a new graph $F$ from $G$ by
contracting the color classes from the dense sets into single vertices, and adding suitable
edges at strategic places in the graph (Lemma~\ref{lem:12}). All these contractions
and edge additions can be easily simulated in $G$, since they involve
pairs of vertices at distance at most 4 apart. The final part will consist in coloring $F$ with $c$ colors, and from this
coloring it will be easy to deduce a $c$-coloring of $G$. Note that
because of the edge additions and contraction, the maximum degree of
$F$ is not bounded by $\Delta$ anymore, but it remains
$O(\Delta)$. The coloring of $F$ is then obtained by a very intricate
semi-random process. Fortunately, for us it boils down to repeated applications of the
Lov\'asz Local Lemma (more precisely, $O((\log \Delta)^{13/12})$ successive
applications), and we just need to make sure that the distributed
Lov\'asz Local Lemma can be safely substituted to its classical
version everywhere in the proof
(Lemma~\ref{lem:13}). With this high-level view in mind, we now proceed with
the proof.

\subsection{Proof of Theorem~\ref{thm:col}} 

Let
$d=10^6\sqrt{\Delta}$. We first compute a $d$-dense decomposition $S,X_1,\ldots,X_t$ of
$G$ in $O(1)$ rounds using Lemma~\ref{lem:decompo}.

\smallskip

A $c$-reducer $D = (C,S')$ is said to be \emph{deletable}
if there are fewer than $c$ vertices in $G-D$ with a neighbor in $S$. Observe that if $D=(C,S')$ is a deletable
$c$-reducer in $G$, then any $c$-coloring of $G-D$ can be
extended to $D$ (since there is a color which does not appear in the
neighborhood of $S'$ in $G-D$). It was observed in~\cite[Observation 8]{MR14} that when
$c\ge \Delta-k_\Delta+1$, any $c$-reducer is deletable. It has the
following consequence.

\begin{lem}\label{lem:8}
  Let $X^r$ be the union of all the $c$-reducers $X_i$. Then there is a
distributed randomized algorithm (running in $G$) that extends any $c$-coloring of
$G-X^r$ to $G$ in $\textsf{T}_{\de+1}(n,\Delta)$ rounds, w.h.p.
\end{lem}

\begin{proof}
For each $c$-reducer $X_i=(C_i,S_i)$, perform the reduction of $G$
with respect to $X_i$ (i.e.\ delete the clique $C_i$, and identify all
the vertices of $S_i$ into a single vertex $v_i$). Let $R$ be the
resulting graph, and let $N$ be the set of newly created vertices in
$R$. Note that the $c$-coloring of $G-X^r$ corresponds to a
$c$-coloring of $R-N$, and our goal is simply to extend this coloring
to $R$ (once this is done, we only have to assign the color of $v_i$
to all the vertices of the stable set $S_i$ in $G$, and to color $C_i$
with the $c-1$ colors distinct from that of $v_i$, which can
clearly be done in $O(1)$ rounds). Since each $X_i$ we consider
here is deletable, each vertex $v_i\in N$ has degree at most $c-1$ in
$R$. It follows from Observation~\ref{obs:1} (similarly as in Section~\ref{sec:thm1}) that the $c$-coloring of $R-N$ can be extended
to $N$ by a distributed randomized algorithm running in $\textsf{T}_{\de+1}(n,\Delta)$ rounds w.h.p., as desired.
\end{proof}

We say that a dense set $X_i$ is \emph{hollow} if $\overline{G[X_i]}$ (the complement of $G[X_i]$) contains a
matching of size at least $100\sqrt{\Delta}$.
We now rephrase Lemma 16 from~\cite{MR14} for our
convenience (the proof of Lemma~\ref{lem:16} follows the same lines as
that of Lemma~\ref{lem:dense}).

\begin{lem}\label{lem:16}
Let $X^h$ be the union of the all the hollow sets $X_i$. Then any $c$-coloring of
$G-X^h$ can be extended to $G$ by a distributed randomized
algorithm running w.h.p.\ in $\textsf{T}_{\de+\Omega(\sqrt{\Delta})}(n,\Delta)$
rounds.
\end{lem}

We temporarily delete from $G$ all the $X_i$ that are $c$-reducers or
hollow. These sets of vertices will be
colored at the very end using Lemmas~\ref{lem:8} and~\ref{lem:16}. Let $H$ be the graph obtained from $G$
by removing the dense
components from Lemmas~\ref{lem:8} and~\ref{lem:16}. Note that the restriction of the
decomposition $S,X_1,\ldots,X_t$ to $H$ is still a $d$-dense
decomposition of $H$, and for convenience we keep denoting it in this
way (even if some sets $X_i$ have disappeared). It follows
from our construction that no dense set $X_i$ in $H$
is a $c$-reducer or is 
such that $\overline{H[X_i]}$ contains a
matching of size at least $100\sqrt{\Delta}$. 

\medskip

Given a subset $Y$ of vertices from some dense component $X_i$, an \emph{external
  neighbor} of $Y$ is a vertex outside of $X_i$ with a neighbor in
$Y$. Recall that a coloring of a graph $G$ partitions the vertex-set
of $G$ into
stable sets, which are called the \emph{color classes} associated to
the coloring. Given a $c$-coloring of $X_i$, we define $C_i$ as the set of
vertices of $X_i$ whose color class is a singleton. We say that a
$c$-coloring of $X_i$ is \emph{nice} if:

\begin{itemize}
\item[(1)] $C_i$ is a clique of size at least $\Delta-2\cdot
  10^6\sqrt{\Delta}$,
\item[(2)] each vertex from any color class of size at least 3 is
  adjacent to all the vertices of $C_i$, and
\item[(3)] if $\{x,y\}$ is a color class of size 2, then either there
  is $z\in C_i$ such that $x,y$ are both adjacent to all the vertices
  of $C_i-\{z\}$, or one of $x,y$ is adjacent to all the vertices of
  $C_i$ and the other is adjacent to all but at most $\tfrac{\Delta}4
  +10^7 \sqrt{\Delta}$ vertices of $C_i$.
\end{itemize}

Note that the unique $c$-coloring of a $c$-reducer is nice.
Lemma~\ref{lem:16} now allows us to
use the following result of~\cite{MR14}. The proof heavily uses the
crucial property
that after the removal of the hollow sets, no dense set $X_i$ contains a large antimatching.

\begin{lem}[Lemmas 19, 20, 21, and 25 in~\cite{MR14}]\label{lem:25}
Each dense set $X_i$ of $H$ has a nice $c$-coloring such that:
\begin{itemize}
\item[(a)] If a color class is not the unique largest colour class in $X_i$, then it has at most $\tfrac{\Delta}2+ 10\sqrt{\Delta}$ external neighbors.
\item[(b)] Every color class of $X_i$ has at most
  $c-\sqrt{\Delta}+3$ external neighbors.
\item[(c)] If there is a colour class of $X_i$ with more than
  $c-10^8\sqrt{\Delta}$ external neighbor, then $|C_i| \ge c - 2 \cdot 10^8$ and each vertex of $C_i$
 has at most $3 \cdot  10^8$ external neighbors.
\item[(d)] If there is a colour class of $X_i$ with more than $c-2\sqrt{\Delta}+3$ external neighbours then $|C_i| = c-1$ and each vertex of $C_i$ has at
 most 5 external neighbors.
\item[(e)] If there is a colour class of $X_i$ with more
than $c - 2\Delta^{3/4} $ external neighbors then $|C_i| \ge c - 5\Delta^{1/4}$ and each
vertex of $C_i$ has at most $8\Delta^{1/4}$ external neighbors.
\end{itemize}
\end{lem}

We stress that the union of the $c$-colorings of each of the dense
components $X_i$ is not necessarily a $c$-coloring of the union of
the dense components: there might be some edges between vertices of
different sets $X_i$ having the same color.
It should be noted that parts (b)--(e) of this result, as stated here, look a bit
different from their counterparts from Lemma 25 in~\cite{MR14}. Indeed, each of properties
(b)--(e) in Lemma 25 from~\cite{MR14} starts by the precondition\emph{``If $X_i$ is not a
reducer or a near-reducer''}. We assumed earlier that $X_i$ is not a
$c$-reducer, so this part of the precondition can certainly be
omitted in our case. A
\emph{$c$-near-reducer} is a subgraph $D$ which is the union of a
clique $C$ of size $c-1$ and a stable set $S'$ of size $\Delta-c+1$, such that each vertex
of $C$ is adjacent to every vertex of $S'$ (in particular each vertex
of $C$ has at most one neighbor outside $D$). Note that each vertex of
$S'$ has at most $\Delta-c+1$ neighbors outside $D$, and thus $S'$ has
at most $(\Delta-c+1)^2$ neighbors outside $D$. Since $c\ge
\Delta-k_\Delta+1$ and $\sqrt{\Delta}-3<k_\Delta=\left \lfloor \sqrt{\Delta+1/4}-3/2 \right \rfloor \leq \sqrt{\Delta+1/4}-3/2$, $S'$ has at most 
$$(\Delta-c+1)^2\le k_\Delta^2\le \Delta-3k_\Delta-2\le c-2k_\Delta
-3\le c-2\sqrt{\Delta}+3$$
neighbors outside $D$. In particular, in our case (i.e.\ when $c\ge \Delta-k_\Delta+1$),
any dense set $X_i$ which is a $c$-near-reducer satisfies
Lemma~\ref{lem:25}(a)--(e), so we can indeed remove the preconditions from
Lemma 25 in~\cite{MR14}. Note also that since each dense set $X_i$ has
diameter at most 2, a nice coloring of each $X_i$ with the additional
properties of Lemma~\ref{lem:25} can be found in $O(1)$ rounds.

\medskip

Based on the nice $c$-coloring of each of the dense components $X_i$
resulting from Lemma~\ref{lem:25}, we now construct (locally) a new graph $F$ from $H$, which will be
easier to color with a semi-random procedure, and such that any $c$-coloring of $F$ can be turned
(locally and efficiently) into a $c$-coloring of $H$.

\begin{lem}[Lemma 12 in~\cite{MR14}]\label{lem:12}
We can construct locally in $H$ in $O(1)$ rounds a graph $F$ of maximum degree at most $10^9\Delta$ (such
that a $c$-coloring of $H$ can be deduced from any $c$-coloring of $F$
in $O(1)$ rounds) and find a partition of the vertices of $F$ into $S, B,
A_1,\ldots, A_t$ such that:
\begin{itemize}
\item[(a)] Every $A_i$ is a clique with
  $c-10^8\sqrt{\Delta}\le|A_i|\le c$.
\item[(b)] Every vertex of $A_i$ has at most $10^8\sqrt{\Delta}$ neighbors in $F - A_i$.
\item[(c)] There is a set $\mathrm{All}_i \subseteq B$ of $c - |A_i|$ vertices which are adjacent to all of $A_i$. Every
other vertex of $F - A_i$ is adjacent to at most $\tfrac34\Delta + 10^8\sqrt{\Delta}$ vertices of $A_i$. 
  \item[(d)] Every vertex of $S$ either has fewer than $\Delta-3\sqrt{\Delta}$ neighbors in $S$ or has at least $900\Delta^{3/2}$ non-adjacent pairs of neighbors within $S$.
\item[(e)] Every vertex of $B$ has fewer than $c-\sqrt{\Delta}+9$
  neighbors in $F - \bigcup_j A_j$. 
\item[(f)] If a vertex $\in B$ has at least $c-\Delta^{3/4}$ neighbors
  in $F-\bigcup_j  A_j$, then there is some $i$ such that: $v$ has at most $c- \sqrt{\Delta}+9$ neighbors in $F -A_i$ and every vertex of $A_i$ has at most $30\Delta^{1/4}$ neighbors in $F - A_i$.
\item[(g)] For every $A_i$, every two vertices outside of $A_i \cup \mathrm{All}_i$ which have at
least $2\Delta^{9/10}$ neighbors in $A_i$ are joined by an edge of $F$.
\end{itemize}
\end{lem}

There is one subtlety in the application of Lemma 12 from~\cite{MR14}:
the statement of Lemma 12 there start with the precondition \emph{``For any
  minimum counterexample,''}. Here we avoid this precondition in the
same way Molloy and Reed avoid it in their application of Lemma 12 in
the algorithmic proof of Theorem 43 from~\cite{MR14} (by starting to remove deletable
reducers and hollow sets).

\medskip

We explain briefly how the graph $F$ is constructed in~\cite{MR14} to stress that
the construction can indeed be performed locally in $H$ (and then in
$G$).

The construction starts by doing the following for each colored dense
component $X_i$. Recall that $C_i$ was defined above as the set of vertices of
$X_i$ whose
color class is a singleton, and it follows from the definition of a
nice coloring that $C_i$ is a clique of size at least $\Delta-2\cdot
  10^6\sqrt{\Delta}$. Now, each color class of size at least 2 (i.e.\
  each color class which is not a singleton)  in $X_i$ is contracted
into a single vertex, and vertices and edges are added inside $X_i$ to
make it into a clique $D_i$ of size precisely $c$. It can be proved
using Lemma~\ref{lem:25} that the maximum degree does not increase too much and that
each clique $D_i$ is not much larger than $C_i$ (see Lemma 29
in~\cite{MR14}). 

\smallskip

A significant issue when trying to find a $c$-coloring of $H$ (or
rather the current modification of $H$) is
that given a clique $D_i$, there might be vertices outside $D_i$ that
have many neighbors (say more than $\tfrac{3\Delta}4$) in $C_i$. Each
such vertex must be in $D_j-C_j$, for some $j\neq i$. Consider such a
vertex $v\in D_j-C_j$, with many neighbors in $C_i$. We need to make
sure that the color of $v$ will be used by one of the few
non-neighbors of $v$ in $D_i$, and one way to do it is, for some vertices
$w \in D_i$, to construct a set $R_w$ of vertices with
many neighbors in $C_i$ such that $\{w\} \cup R_w$ is a stable  set
and every vertex with many neighbors in $C_i$ lies in such a set
$R_w$. We then contract each set $\{w\} \cup R_w$ into a single
vertex (this will force that all these vertices have the same color at
the end), and denote by $A_i$ the set $C_i$ after the removal of the
vertices $w$ for which some set $R_w$ was defined. We also set $\mathrm{All}_i=D_i-A_i$. Again it can be proved that the
maximum degree does not increase too much and each $A_i$ is not too
small compared to $C_i$ (see Lemma 30
in~\cite{MR14}). 

\smallskip

A second issue (related to the issue described above) is that we
need to prevent that many different external neighbors of $A_i$ are
all colored with the same color, and their neighborhoods cover $A_i$
(this would prevent this color from being used in $A_i$). The way it
is solved in~\cite{MR14} is by adding an edge between every pair of external
neighbors of $A_i$ having at least $\Delta^{9/10}$ neighbors in
$A_i$. It is proved (see Lemma 31 in~\cite{MR14}) that it does not
increase the maximum degree too much and is enough to deduce
properties (a)--(g) of
Lemma~\ref{lem:12} (the issue raised in this paragraph is in
particular related to property (g)).

\smallskip

To sum up, $F$ has been obtained from $H$ by identifying (or adding edges
between) pairs of vertices at distance at most 4, since each dense
component has diameter at most 2 and any two vertices that have been
identified or joined by an edge have a neighbor in the same dense component. Moreover, each
modification has been carried out independently by each dense set
$X_i$ (even if the modifications had some impact outside of $X_i$), so
$F$ can be simulated by
$H$ (and then by $G$) with at most a small multiplicative loss on the round complexity. It is also clear that a $c$-coloring of $H$ can
be obtained from any $c$-coloring of $F$ in $O(1)$ rounds.

\medskip

It remains to show how to efficiently color $F$ with $c$ colors.

\begin{lem}\label{lem:13}
The graph $F$ described in Lemma~\ref{lem:12} can be colored with $c$
colors in $\textsf{T}_{\de+\Omega(\sqrt{\Delta})}(n,\Delta) +O((\log\Delta)^{13/12})\cdot \textsf{T}_{LLL}(n,\poly
\Delta)$ rounds, w.h.p.
\end{lem}

We will be rather brief here (the proof of the corresponding
sequential statement, Lemma 13 in~\cite{MR14}, takes 20 pages). Consider some $1\le i \le t$. Since
 $\mathrm{All}_i\cup A_i$ forms a clique of size $c$, we need to make
 sure that the colors that do not appear in $\mathrm{All}_i $ do not
 appear either on too many external neighbors of $A_i$. A key property of the construction of
 $F$ (see  properties (c) and (g) in Lemma~\ref{lem:12}) is
 that for any color $x$, there is at most one vertex $v\not\in
 \mathrm{All}_i\cup A_i$ having at least $2\Delta^{9/10}$ neighbors in
 $A_i$ that is colored $x$, and moreover $v$ has at most
 $\tfrac34\Delta+o(\Delta)$ neighbors in $A_i$. The goal will be to
 maintain this property throughout the whole process, namely that all
 of the time, at most $\tfrac34\Delta+o(\Delta)$ vertices of $A_i$ have a
 neighbor colored $x$ outside of $\mathrm{All}_i\cup A_i$ (let us call
 this event $E(i,x)$).

\smallskip

The starting point will be to color $S$ (the $d$-sparse vertices, see
property (d) of Lemma~\ref{lem:12})
randomly as in the proof of Lemma~\ref{lem:sparse}, i.e.\ with the
property that many colors are repeated in the neighborhoods of the
high degree vertices, but also with the
additional property that $E(i,x)$ still holds for any $i,x$ after the
coloring. Similarly as in the proof of Theorem~\ref{thm:1}, coloring $S$ takes
$\textsf{T}_{LLL}(n,\poly \Delta)+\textsf{T}_{\de+\Omega(\sqrt{\Delta})}(n,\Delta)$
rounds, w.h.p. (after a single application of the distributed
Lov\'asz Local Lemma, coloring the uncolored vertices of $S$ is an
instance of a $(\de+\Omega(\sqrt{\Delta}))$-list-coloring
problem). Note that another part of the algorithm involves solving
instances of the harder $(\de+1)$-list-coloring
problem (Lemmas~\ref{lem:8}), so this part is
dominated by the other parts of the algorithm.

We then proceed to extend the coloring to $B$. Recall that by property (e) of
Lemma~\ref{lem:12}, each vertex of $B$ has at most
$c-\Omega(\sqrt{\Delta})$ neighbors in $F - \bigcup_j A_j$. It turns
out that it is a bit too high to randomly extend the coloring of $S$
to $B$ while maintaining property $E(i,x)$, so instead we color the
remaining vertices in this order:

\begin{enumerate}
\item We first color the set $B_H$ of vertices of $B$ with at
most $c-\Delta^{3/4}$ neighbors in $F - \bigcup_j A_j$ (coloring these vertices
will preserve $E(i,x)$). 
\item We then color the sets $A_i$ such that each vertex
of $A_i$ has at most $30\Delta^{1/4}$ neighbors outside of
$\mathrm{All}_i\cup A_i$.
\item We color
$B_L=B-B_H$, using property (f) of
Lemma~\ref{lem:12} (which implies that property $E(i,x)$ can now be
preserved while coloring these vertices). 
\item Finally we color the sets $A_i$ that have not been colored yet.
\end{enumerate}

The proofs that desirable properties are maintained during the
coloring of the vertices of $S$ and $B$ and the $A_i$ are fairly
similar to the proof of Lemma~\ref{lem:sparse}, in the sense that they
boil down to the estimation of the expectation of some random
variables, the proof that these random variables are highly
concentrated, and then some application of the Lov\'asz Local Lemma.

We should note two important differences, though. 

\begin{itemize}
\item The first is that
instead of a single random partial coloring, followed by  a
 greedy procedure completing the coloring, the procedure
for coloring $S$, $B_H$, and $B_L$
here involves multiple rounds (more specifically, at most $O((\log \Delta)^{13/12})$
rounds, w.h.p.) of random partial coloring and a careful
study of all the random variables throughout the process. 
\item The second
is that while coloring the $A_i$, the partial random coloring
procedure is a bit
different than in the proof of  Lemma~\ref{lem:sparse}. Recall that
each $A_i$ is a clique, so
assigning each vertex a color uniformly at random, and then uncoloring
pairs of vertices with the same color would be extremely
unpractical. Instead, each $A_i$ is colored with a permutation of the
$|A_i|$ colors not appearing on $\mathrm{All}_i$, taken uniformly at
random among all the possible permutations. A consequence is that
instead of using Talagrand's Inequality to prove the concentration of
random variables around their expectation, McDiarmid's
Inequality has to be used instead (see~\cite{MR14}), but the resulting
bounds are of a similar order of magnitude.
\end{itemize}

It can be checked that in all the applications of the Lov\'asz Local
Lemma in~\cite{MR14}, bad events correspond to subgraphs of $H$ of
bounded radius, and the probabilities of the bad events are
smaller than any fixed polynomial function of the maximum degree of
the event dependency graph (these probabilities are typically of order
$\exp(-d^{\alpha})$ or $\exp(-\beta\log^2
d)$, where $\alpha,\beta>0$ and
$d$ is the maximum degree of the event dependency graph),
so in particular any polynomial criterion is satisfied and we can 
substitute the distributed Lov\'asz Local Lemma everywhere
in the proof, and since the semi-random process involves at most
$O((\log \Delta)^{13/12})$
successive applications of the Lov\'asz Local Lemma\footnote{In the
  proof of Molloy and Reed~\cite{MR14} the authors use $O(\Delta^\lambda)$
  successive applications of the Local Lemma (for any fixed constant $\lambda>0$), but
the proof can easily be optimized to work with only $O((\log
\Delta)^{13/12})$ applications of the Local Lemma. The bound $\log^{13/12} \Delta$ comes
from the proof of the concentration of $Z_C'$, page 175
of~\cite{MR14}, which dominates the other related bounds on the number $I$
of iterations in the proof of Lemma 34 of~\cite{MR14}. Note that the
authors of~\cite{MR14} were aiming at a polynomial complexity, so it
did not make much sense for them to replace the polynomial number of iterations
by a polylogarithmic number of iterations, at the cost of tedious computations.}, a $c$-coloring of $F$ can
be obtained in $\textsf{T}_{\de+\Omega(\sqrt{\Delta})}(n,\Delta) +O((\log
\Delta)^{13/12})\cdot \textsf{T}_{LLL}(n,\poly \Delta)$ rounds, w.h.p.

\medskip

We find it necessary to insist on a technical (but important) detail
here. Theorem~\ref{thm:LLL} uses the so-called
\emph{variable setting} of the Local Lemma, which covers most
applications of the original Local Lemma but not all of them. In
particular we have to be careful here since the coloring of the
$A_i$ involved random permutations of colors assigned to a given set
of vertices, instead of colors
chosen uniformly at random for each vertex, and it is not clear at first sight whether the
former can be handled in the variable setting. It turns out that 
it can, since in the proof of Lemmas 39 and 40 in~\cite{MR14} the
graph under consideration has one vertex for each uncolored $A_i$, and
an edge between two vertices if the corresponding sets $A_i$ are
adjacent in $H$ (since each set $A_i$ is a clique, this graph can be
simulated within $H$). The variable associated to each vertex is the
random permutation of colors assigned to the corresponding set $A_i$,
so this is indeed an instance of the variable setting of the Local
Lemma, and we can use Theorem~\ref{thm:LLL}.

\medskip

Now that $F$ has been colored with $c$ colors, we obtain a $c$-coloring of $H$ in $O(1)$ rounds using Lemma~\ref{lem:12},
and it remains to color the dense
components $X_i$ that are $c$-reducers, or such that $\overline{G[X_i]}$ contains a
matching of size at least $100\sqrt{\Delta}$ (recall that these dense components
had been removed from the graph at the beginning of the
procedure). It follows from Lemmas~\ref{lem:8} and~\ref{lem:16} that
the $c$-coloring of $H$ can be extended to the remaining dense
components of $G$ w.h.p.\ in $T_{\de+1}(n,\Delta)$
rounds. Hence, the overall round complexity of the algorithm is $O(T_{\de+1}(n,\Delta))+O((\log
\Delta)^{13/12})\cdot \textsf{T}_{LLL}(n,\poly \Delta)$, w.h.p. Using
Theorems~\ref{thm:lcol1} and~\ref{thm:LLL}, this is $\min\{O((\log\Delta)^{1/12}\log n), 2^{O(\log \Delta+\sqrt{\log \log
    n})}\}$ rounds w.h.p., which concludes the proof of Theorem~\ref{thm:col}.\hfill
$\Box$

\subsection{Summary of our contributions}

We now make a brief summary of our contributions (to make
clear what we added and subtracted from the proof of Molloy and
Reed~\cite{MR14}).

\smallskip

In~\cite{MR14}, $c$-reducers are dealt with slightly differently: some are
simply removed as we do here, but some are reduced as in the
definition of $c$-reduction of Section~\ref{sec:reducer} (i.e.\ by removing the clique and
contracting the stable set into a single vertex). This operation can
create new $c$-reducers, and thus $c$-reducers have to be reduced sequentially
until no $c$-reducer appears in the graph (the fact that it has to be
done sequentially is essentially the proof of Theorem~\ref{thm:sharp}). For $c$-near-reducers, the
situation is slightly more complicated (see Lemma 27 in~\cite{MR14}) but again
inherently sequential. It is fortunate that in our case (i.e.\ when $c\ge
\Delta-k_\Delta+1$), we do not need to worry about these cases, as
explained after Lemma~\ref{lem:25}. So our
contribution is simply to have checked that the initial $d$-dense
decomposition can be computed locally (see Lemma~\ref{lem:decompo}), that the construction of $F$
can be performed locally, that all the applications of the Local Lemma
can be also carried out locally in the phase where the $c$-coloring of
$F$ is obtained, and that the resulting coloring of $H$ can be
extended to $G$ locally and efficiently (see
Lemmas~\ref{lem:8} and~\ref{lem:16}).

\section{Concluding remarks}\label{sec:conc}

Note that using recent results of Ghaffari \emph{et
  al.}~\cite{GHK17}, the randomized algorithms in Theorem~\ref{thm:col}
and~\ref{thm:1} can be replaced by deterministic algorithms with a
round complexity of $2^{O(\log \Delta+\sqrt{\log n})}$. An interesting question is whether the dependency in
  $\Delta$ can be significantly reduced (the same question can be
  asked for Theorem~\ref{thm:1}
and~\ref{thm:col}). It seems to us that techniques that have been developed
so far, such as Theorem 1.8 in~\cite{GHK17} or the ad-hoc techniques from~\cite{FG17}, do not work well in our case.

\medskip

When the maximum degree $\Delta$ is a constant, the $(\de+1)$-list coloring
problem can be solved in $O(\log^*n)$ rounds~\cite{GPS88,Lin92}, which is much faster than the
round complexity of Theorems~\ref{thm:lcol1} and~\ref{thm:lcolp}. In
this case it is interesting to use a slightly faster version of
Theorem~\ref{thm:LLL} from~\cite{GHK17}, with round complexity
$\exp(\exp(O(\sqrt{\log \log \log n})))$, or $\exp(\exp(\exp(O(\sqrt{\log
  \log \log \log n}))))$, or more generally
$\exp^{(i)}(O(\sqrt{\log^{(i+1)}n}))$ for any $1\le i \le
\log^*n-2\log^* \log^*n$. It is not difficult to see that in this case
this round
complexity dominates the other parts of the algorithms used in this
paper. It follows that the round complexity in Theorem~\ref{thm:1}
and~\ref{thm:col} in the bounded degree case can be replaced by $\exp^{(i)}(O(\sqrt{\log^{(i+1)}n}))$ for any $1\le i \le
\log^*n-2\log^* \log^*n$. Moreover, any improvement on the round
complexity of the distributed Lov\'asz Local Lemma under some polynomial
criterion would immediately yield an improved complexity in Theorems~\ref{thm:col}
and~\ref{thm:1} in the case of bounded degree graphs.

\medskip

We have proved that the threshold between efficient tractability and
intractability of finding an optimal coloring of a graph of (sufficiently
large) maximum degree in the \textsf{LOCAL} model occurs at
$c=\Delta-k_\Delta+1$ colors (for all values of $\Delta$), or when
$c=\Delta-k_\Delta$ and $(k_\Delta+1)(k_\Delta+2)=\Delta$. So a
natural question is the status of the round complexity of obtaining an optimal
coloring when $c=\Delta-k_\Delta$ and
$(k_\Delta+1)(k_\Delta+2)<\Delta$. We have no clear idea of what the
right answer should be, but in this case we can at least decide if the
chromatic number is at most $c$ in $O(\Delta^{5/2})$ rounds (deterministically),
using Corollary 7c(ii) in~\cite{MR14}, which says that in this case
we only need
to check the $c$-colorability of connected subgraphs of size
$O(\Delta^{5/2})$, which can be done in $O(\Delta^{5/2})$ rounds in
the \textsf{LOCAL} model of computation.

\begin{acknowledgement} 
We thank David Harris for pointing out the updated version
of~\cite{GHK17} and for his kind remarks on earlier versions of the
paper. We also thank two anonymous reviewers for their detailed
comments and suggestions.
\end{acknowledgement}

\end{document}